\documentclass[12pt,reqno]{amsart}
\usepackage{amsmath,amssymb,amsfonts,amscd,latexsym,amsthm,mathrsfs,verbatim}
\newtheorem{lemma}{Lemma}
\newtheorem{theorem}{Theorem}

\theoremstyle{remark}
\newtheorem*{remark}{\bf Remark}

\let\wt\widetilde
\newcommand\m{\mathrm{m}}

\renewcommand{\d}{{\mathrm d}}
\newcommand{\CT}{\operatorname{CT}}

\renewcommand{\Im}{\operatorname{Im}}
\renewcommand{\Re}{\operatorname{Re}}
\newcommand{\sign}{\operatorname{sign}}
\begin{document}

\title{On the Mahler measure of~a~family~of~genus~2~curves}

\author{Marie Jos\'e Bertin}
\address{Universit\'e Pierre et Marie Curie (Paris 6), Institut de Math\'ematiques, 4 Place Jussieu, F-75252 Paris, FRANCE}
\email{marie-jose.bertin@imj-prg.fr}

\author{Wadim Zudilin}
\address{School of Mathematical and Physical Sciences, The University of Newcastle, Callaghan NSW 2308, AUSTRALIA}
\email{wzudilin@gmail.com}

\date{17 May 2014 \& 30 December 2015}

\begin{abstract}
We establish a general identity between the Mahler measures $\m(Q_k(x,y))$ and $\m(P_k(x,y))$ of two polynomial families,
where $Q_k(x,y)=0$ and $P_k(x,y)=0$ are generically hyperelliptic and elliptic curves, respectively.
\end{abstract}


\thanks{The second author is supported by Australian Research Council grant DP140101186.}

\maketitle

\section{Introduction}
\label{s-intro}

In his 1998 \emph{Experimental Mathematics} paper \cite{Boy98} Boyd gave many conjectures about (logarithmic) Mahler measures
of two-variable polynomials,
$$
\m(P(x,y))
=\frac1{(2\pi i)^2}\iint_{|x|=|y|=1}\log|P(x,y)|\,\frac{\d x}x\,\frac{\d y}y,
$$
the conjectures that started the whole area of research on the border
of number theory, $K$-theory, algebraic geometry and analysis. One of the exemplar families
include evaluations of $\m(P_k(x,y))$ for
$$
P_k(x,y)=(x+1)y^2+(x^2+kx+1)y+(x^2+x), \quad k\in\mathbb Z,
$$
which can be characterised by $\m(P_k)/L'(E,0)\in\mathbb Q^\times$, where $E:P_k(x,y)=0$ is (generically)
an elliptic curve and $L'(E,0)$ is the derivative of its $L$-function $L(E,s)$ at $s=0$. Very few of these are proven so far:
$k=0,6$ (conductor~36) by Rodriguez-Villegas~\cite{RV99},
$k=1,10,-5$ (conductor~14) by Mellit~\cite{Mel09}, and
$k=4,-2$ (conductor~20) by Rogers and the second author~\cite{RZ12}.
Note that these particular cases (as well as all other proven cases of elliptic type,
for both CM and non-CM curves) accidentally fall under application of the Mellit--Brunault formula~\cite{Zud14}.

Boyd also indicated in \cite{Boy98} some families of Mahler measures related to genus~2 curves.
One of them was studied in great detail in Bosman's thesis \cite{Bos04}:
namely, he considered the family $\m(Q_k(X,Y))$, where
$$
Q_k(X,Y)=Y^2+(X^4+kX^3+2kX^2+kX+1)Y+X^4,
$$
and showed that the curve $C:Q_k(X,Y)=0$ has genus $2$ for $k\in\mathbb C\setminus\{-1,0,4,8\}$;
genus $1$ for $k=-1$ and~$4$; genus~$0$ for $k=8$; and reducible to two irreducible components of genus~$0$ for $k=0$.
The principal result of~\cite{Bos04} includes the following three evaluations:
\begin{gather}
\m(Q_2)=L'(E_{36},0), \quad\text{where}\; E_{36}:y^2=x^3+1 \;\;\text{(conductor 36)},
\label{Q2}
\\
\m(Q_{-1})=2L'(\chi_{-3},-1)
\quad\text{and}\quad
\m(Q_8)=4L'(\chi_{-4},-1).
\nonumber
\end{gather}
In fact, the other genus~1 case $k=4$ admits a parametrisation by modular units (of level $N=20$),
so that Theorem~1 from~\cite{Zud14} applies to produce the evaluation
\begin{equation}
\m(Q_4)=4L'(E_{20},0), \quad\text{where}\; E_{20}:y^2=x^3+x^2+4x+4 \;\;\text{(conductor 20)},
\label{Q4}
\end{equation}
conjectured by Boyd in~\cite{Boy98}.

By analysing Boyd's Tables~2 and~10 in \cite{Boy98}, as well as Bosman's reduction of the hyperelliptic curve $C$ in \cite[Chap.~8]{Bos04},
one can recognise a clear link between the families $Q_k(X,Y)$ and $P_{2-k}(x,y)$.
In fact, the Jacobian of $C:Q_k(X,Y)=0$ generically splits into two elliptic curves, the first being birationally equivalent
to $E:P_{2-k}(x,y)=0$. This gives birth to a relation between the corresponding Mahler measures.

\begin{theorem}
\label{th1}
The following is true for real values of $k$:
\begin{equation*}
\m(Q_k)=\begin{cases}
2\m(P_{2-k}) & \text{for $0\le k\le4$}, \\
\phantom1\m(P_{2-k}) & \text{for $k\le-1$}.
\end{cases}
\end{equation*}
\end{theorem}

Bosman comments in \cite{Bos04} on his proof of evaluations of $\m(Q_2)$,
$\m(Q_{-1})$ and $\m(Q_8)$:
\begin{quote}
``Although I succeeded in proving these identities, I still don't understand why such identities must
be valid. These proofs give a deduction which one can verify step-by-step but they seem to exist
coincidentally. It is not at all clear how the arithmetic structure of the polynomials is related to
the arithmetic structure of the $L$-series. We will rewrite the integral to so-called dilogarithm sums
and these we will rewrite in terms of $L$-series. It is an open problem whether this is always possible.
In the elliptic curve case it is not clear at all how to prove a relation between the dilogarithm sums
and the $L$-series, except for a few instances.''
\end{quote}
We believe that Theorem~\ref{th1} and its proof below provides us with at least a partial
understanding of the magic behind the $L$-series evaluations of Mahler measures. The results
in \cite{Mel09,RV99,RZ12} and the theorem also lead to different proofs of the formulae for $\m(Q_2)$ and $\m(Q_{-1})$,
as well as to four more equalities conjectured by Boyd in~\cite{Boy98}.

\begin{theorem}
\label{th2}
The evaluations \eqref{Q2}, \eqref{Q4} and
\begin{alignat*}{2}
\m(Q_1)&=2\m(P_1)=2L'(E_{14},0), &\quad&\text{where}\; E_{14}:y^2+xy+y=x^3-x,
\\
\m(Q_{-2})&=\m(P_4)=3L'(\hat E_{20},0), &\quad&\text{where}\; \hat E_{20}:y^2=x^3+x^2-x,
\\
\m(Q_{-4})&=\m(P_6)=2L'(\hat E_{36},0), &\quad&\text{where}\; \hat E_{36}:y^2=x^3-15x+22,
\\
\m(Q_{-8})&=\m(P_{10})=10L'(\hat E_{14},0), &\quad&\text{where}\; \hat E_{14}:y^2+xy+y=x^3-11x+12,
\end{alignat*}
are valid.
\end{theorem}

\section{Bosman's family of genus 2 curves}
\label{s-bosman}

The polynomials
$$
Q_k(X,Y)=Y^2+(X^4+kX^3+2kX^2+kX+1)Y+X^4
$$
define generically the family of curves $Q_k(X,Y)=0$ of genus~2. The split of its
Jacobian into two families of elliptic curves correspond to a reduction of hyperelliptic integrals,
a phenomenon first observed in general settings by Goursat~\cite{Gou85}.
Since
$Q_k(X,Y)=X^4\wt Q_k(X,Y/X^2)$ where
$$
\wt Q_k(X,Y)=Y^2+\biggl(\biggl(X+\frac1X\biggr)^2+k\biggl(X+\frac1X\biggr)+2k-2\biggr)Y+1,
$$
we conclude that $\m(Q_k(X,Y))=\m(\wt Q_k(X,Y))$. We now investigate where the curve $\wt Q_k(X,Y)=0$
cuts the torus $|X|=|Y|=1$.

On $|X|=1$, both
$$
B(X):=B_k(X)=\biggl(X+\frac1X\biggr)^2+k\biggl(X+\frac1X\biggr)+2k-2
$$
and
\begin{align}
\Delta(X)
&:=\Delta_k(X)=B(X)^2-4
\nonumber\\
&\phantom:=\biggl(\biggl(X+\frac1X\biggr)^2+k\biggl(X+\frac1X\biggr)+2k\biggr)
\biggl(X+\frac1X+2\biggr)\biggl(X+\frac1X+k-2\biggr)
\label{Delta}
\end{align}
are real-valued. Then
$$
\wt Q_k(X,Y)=Y^2+B(X)Y+1=(Y-Y_1(X))(Y-Y_2(X)),
$$
where
$$
Y_1(X),Y_2(X)=\frac{-B(X)\pm\sqrt{\Delta(X)}}2.
$$
By Vi\`ete's theorem $Y_1(X)Y_2(X)=1$; therefore, we get $|Y_1(X)|=|Y_2(X)|=1$ for $\Delta(X)\le0$.
If $\Delta(X)>0$, then both zeroes $Y_1(X)$ and $Y_2(X)$ are real of the same sign as $-B(X)$.
In this case we numerate them in such a way that
$$
|Y_2(X)|<1<|Y_1(X)|;
$$
in other words,
$$
Y_1(X)=\frac{-B(X)-\sign(B(X))\sqrt{\Delta(X)}}2.
$$

\begin{lemma}
\label{lem1}
If $\Delta(X)>0$ on the unit circle, then signs of $B(X)$ and $k$ coincide.
\end{lemma}

\begin{proof}
Taking $X=e^{i\theta}$, where $-\pi<\theta<\pi$, so that $X+1/X=2\cos\theta$ is between $-2$ and~$2$,
we see from \eqref{Delta} that the sign of $\Delta(X)$ is fully controlled by the signs of
\begin{gather*}
B(X)+2=\biggl(X+\frac1X\biggr)^2+k\biggl(X+\frac1X\biggr)+2k=4\cos^2\theta+2k(1+\cos\theta)
\\ \text{and}\quad
X+\frac1X+k-2=2\cos\theta+k-2.
\end{gather*}

If $k>0$, the former expression is strictly positive; hence $\Delta(X)=(B(X)-2)\*(B(X)+2)>0$ implies
$B(X)-2>0$, in particular, $B(X)>0$.

If $k<0$, then $2\cos\theta+k-2<0$, so that $\Delta(X)>0$ is equivalent to $B(X)+2<0$, hence $B(X)<0$.
\end{proof}

Using Jensen's formula and the symmetry $Y_1(X)=Y_1(X^{-1})$, we get
\begin{align*}
\tilde g(k)
&=\m(Q_k(X,Y))=\m(\wt Q_k(X,Y))
\\
&=\frac1{(2\pi i)^2}\iint_{|X|=|Y|=1}\log|\wt Q_k(X,Y)|\,\frac{\d X}X\,\frac{\d Y}Y
\displaybreak[2]\\
&=\frac1{2\pi i}\int_{|X|=1}\log|Y_1(X)|\,\frac{\d X}X
\displaybreak[2]\\
&=\frac1{\pi i}\int_{\substack{|X|=1\\\Im X>0}}\Re\log Y_1(X)\,\frac{\d X}X
\displaybreak[2]\\
&=\frac1\pi\,\Re\int_0^\pi\log Y_1(e^{i\theta})\,\d\theta
\\
&=\frac1\pi\,\Re\int_0^\pi\log\frac{B(e^{i\theta})+\sign(k)\sqrt{\Delta(e^{i\theta})}}2\,\d\theta,
\quad k\ne0.
\end{align*}

\begin{lemma}
\label{lem2}
For $k<-1$ and $k>0$,
$$
\frac{\d\tilde g(k)}{\d k}
=\frac{\sign(k)}\pi\,\Re\int_{-\infty}^{k(3-k)}\frac{\d v}{\sqrt{-(v+4)(v^2+k(k-4)v+4k^2)}}.
$$
\end{lemma}

\begin{proof}
We have
\begin{align*}
\frac{\d}{\d k}\,\log\frac{B\pm\sqrt{\Delta}}2
&=\frac{\d}{\d B}\,\log\frac{B\pm\sqrt{B^2-4}}2\cdot\frac{\d B}{\d k}
=\mp\frac1{\sqrt{B^2-4}}\cdot\biggl(X+\frac1X+2\biggr)
\\
&=\mp\frac1{\sqrt{\Delta}}\cdot\biggl(X+\frac1X+2\biggr).
\end{align*}
Therefore,
$$
\frac{\d\tilde g(k)}{\d k}
=-\frac{\sign(k)}\pi\,\Re\int_0^\pi\frac{2(\cos\theta+1)\,\d\theta}{\sqrt{\Delta_k(e^{i\theta})}}.
$$
Changing to $t=\cos\theta=(X+1/X)/2$ we can write the result as
\begin{align*}
\frac{\d\tilde g(k)}{\d k}
&=\frac{\sign(k)}\pi\,\Re\int_{-1}^1\frac{t+1}{\sqrt{(2t^2+kt+k)(t+1)(2t+k-2)}}\,\frac{\d t}{\sqrt{1-t^2}}
\\
&=\frac{\sign(k)}\pi\,\Re\int_{-1}^1\frac{\d t}{\sqrt{(1-t)(2t^2+kt+k)(2t+k-2)}}.
\end{align*}
For $k<-1$ and $k>0$ we further employ the substitution $t=(v+2k^2-2k)/(v-4k)$ to get the desired form.
\end{proof}

We remark that for $k>4$ the formula of Lemma~\ref{lem2} translates into
$$
\frac{\d\tilde g(k)}{\d k}
=\frac1\pi\int_{-\infty}^{k(3-k)}\frac{\d v}{\sqrt{-(v+4)(v^2+k(k-4)v+4k^2)}}.
$$
The result is an \emph{incomplete} elliptic integral, and this fact creates a natural obstruction
for having no relations between the two Mahler measures in Theorem~\ref{th1} when $k>4$.

\section{Boyd's family of elliptic curves}
\label{s-boyd}

Let us now turn our attention to an ``easier'' object\,---\,the family $P_k(x,y)=0$ of elliptic curves (generically), where
\begin{align*}
P_k(x,y)
&=(x+1)y^2+(x^2+kx+1)y+(x^2+x)
\\
&=(x+1)(y+1)(x+y)-(2-k)xy.
\end{align*}
Denote $g(k)=\m(P_{2-k}(x,y))$. Since $P_k(x^2,y^2)=x^2y^2\wt P_{2-k}(x,y)$ where
$$
\wt P_k(x,y)=\biggl(x+\frac1x\biggr)\biggl(y+\frac1y\biggr)\biggl(\frac xy+\frac yx\biggr)-k,
$$
we also have $g(k)=\m(\wt P_k(x,y))$. Then
$$
\wt P_{-1}(x,y)=\biggl(x+\frac1x\biggr)\biggl(y+\frac1y\biggr)\biggl(\frac xy+\frac yx\biggr)+1
$$
is real-valued and ranges between $-7$ and $9$ on the torus $|x|=|y|=1$;
therefore, for $|k|>9$ we can write the Mahler measure
\begin{align*}
\m(\wt P_k(x,y))
&=\frac1{(2\pi i)^2}\iint_{|x|=|y|=1}\log|k+1-\wt P_{-1}(x,y)|\,\frac{\d x}x\,\frac{\d y}y
\\
&=\log|k+1|+\frac1{(2\pi i)^2}\iint_{|x|=|y|=1}\Re\log\biggl(1-\frac{\wt P_{-1}(x,y)}{k+1}\biggr)\,\frac{\d x}x\,\frac{\d y}y
\displaybreak[2]\\
&=\log|k+1|-\Re\sum_{n=1}^\infty\frac1n\,\frac1{(2\pi i)^2}\iint_{|x|=|y|=1}\frac{\wt P_{-1}(x,y)^n}{(k+1)^n}\,\frac{\d x}x\,\frac{\d y}y
\\
&=\Re\biggl(\log(k+1)-\sum_{n=1}^\infty\frac1n\,\frac{\CT(\wt P_{-1}(x,y)^n)}{(k+1)^n}\biggr),
\end{align*}
where the constant term
$$
\CT(\wt P_{-1}(x,y)^n)
=\sum_{\substack{j,\ell,m\ge0\\j+\ell+m=n}}\biggl(\frac{n!}{j!\,\ell!\,m!}\biggr)^2
=\sum_{j=0}^n{\binom nj}^2\binom{2j}j.
$$
Differentiating the resulting expression for $|k|>9$, we obtain
\begin{align*}
\frac{\d g(k)}{\d k}
=\Re\sum_{n=0}^\infty(k+1)^{-n+1}\sum_{j=0}^n{\binom nj}^2\binom{2j}j.
\end{align*}
The series
$$
f(z)=\sum_{n=0}^\infty z^n\sum_{j=0}^n{\binom nj}^2\binom{2j}j
$$
satisfies the Picard--Fuchs differential equation \cite{Ber01,Ver01}
$$
z(z-1)(9z-1)\frac{\d^2f}{\d z^2}+(27z^2-20z+1)\frac{\d f}{\d z}+3(3z-1)f=0,
$$
whose singularities are at $z=0,1/9,1$ and $\infty$. Therefore, our result for $\d g(k)/\d k$
can be extended on the real line by continuity to the intervals $k<-1$ and $k>8$ using
$$
f(z)=\sum_{m=0}^\infty\frac{(3m)!}{m!^3}\,\frac{z^{2n}(1-z)^n}{(1-3z)^{3n+1}}
=\frac1{1-3z}\,{}_2F_1\biggl(\begin{matrix} \frac13, \, \frac23 \\ 1 \end{matrix}\biggm| \frac{27z^2(1-z)}{(1-3z)^3} \biggr)
\quad\text{for}\;\; z<0
$$
and
$$
f(z)=\sum_{m=0}^\infty\frac{(3m)!}{m!^3}\,\frac{z^n(1-z)^{2n}}{(1+3z)^{3n+1}}
=\frac1{1+3z}\,{}_2F_1\biggl(\begin{matrix} \frac13, \, \frac23 \\ 1 \end{matrix}\biggm| \frac{27z(1-z)^2}{(1+3z)^3} \biggr)
\quad\text{for}\;\; 0<z<\frac19,
$$
respectively (see, e.g., \cite[Section~3]{BSWZ12}). Namely, we have
\begin{equation}
\frac{\d g(k)}{\d k}
=\frac1{k-2}\,{}_2F_1\biggl(\begin{matrix} \frac13, \, \frac23 \\ 1 \end{matrix}\biggm| \frac{27k}{(k-2)^3} \biggr)
\quad\text{for}\;\; k<-1
\label{S2}
\end{equation}
and
\begin{equation}
\frac{\d g(k)}{\d k}
=\frac1{k+4}\,{}_2F_1\biggl(\begin{matrix} \frac13, \, \frac23 \\ 1 \end{matrix}\biggm| \frac{27k^2}{(k+4)^3} \biggr)
\quad\text{for}\;\; k>8.
\label{S3}
\end{equation}
In fact, the latter formula (and slightly more) are shown in \cite[Section~4.2]{RZ12} to be true for $2<k<8$.

\begin{lemma}
\label{lem3}
For $0<k<8$,
\begin{equation*}
\frac{\d g(k)}{\d k}
=\frac1{2\pi}\int_0^1\frac{\d t}{\sqrt{t(1-t)(k^2t^2+(4-k)kt+4)}}.
\end{equation*}
\end{lemma}

\begin{proof}
This integral representation is established in \cite[Section~4.2]{RZ12} for the range $2<k<8$. Since the both sides
are continuous on the broader interval $0<k<8$, the lemma follows.
\end{proof}

\begin{lemma}
\label{lem4}
For $k<-1$,
$$
\frac{\d g(k)}{\d k}
=-\frac{1}{2\pi}\int_0^1\frac{p(1+p)\,\d t}{\sqrt{t(1-t)(1+2p-p^3(2+p)t)}}
\quad\text{for}\;\; k=-\frac2{p(1+p)},
$$
where $0<p<1$.
\end{lemma}

\begin{proof}
As in the proof of \cite[Lemma~9]{RZ12} we use Ramanujan's transformation \cite[p.~112, Theorem~5.6]{Be98}
\begin{align*}
\frac1{1+p+p^2}\,{}_2F_1\biggl(\begin{matrix} \frac13, \, \frac23 \\
1 \end{matrix}\biggm| \frac{27p^2(1+p)^2}{4(1+p+p^2)^3} \biggr)
&=\frac1{\sqrt{1+2p}}\,{}_2F_1\biggl(\begin{matrix} \frac12, \, \frac12 \\
1 \end{matrix}\biggm| \frac{p^3(2+p)}{1+2p} \biggr)
\\
&=\frac1{\pi}\int_0^1\frac{\d t}{\sqrt{t(1-t)(1+2p-p^3(2+p)t)}},
\end{align*}
valid for $0\le p<1$. The required formula follows from this transformation by choosing
$p=(\sqrt{1-8/k}-1)/2$ in \eqref{S2}.
\end{proof}

\section{Comparison}
\label{s-finale}

\begin{lemma}
\label{lem5}
For the derivatives of Mahler measures defined in Sections~\textup{\ref{s-bosman}} and~\textup{\ref{s-boyd}},
$$
\frac{\d\tilde g(k)}{\d k}
=\frac{\d g(k)}{\d k} \quad\text{if}\; k<-1,
\quad\text{and}\quad
\frac{\d\tilde g(k)}{\d k}
=2\frac{\d g(k)}{\d k} \quad\text{if}\; 0<k\le4.
$$
\end{lemma}

\begin{proof}
If $k<-1$, we let $k=-2/(p(1+p))$ where $0<p<1$ and write the result of Lemma~\ref{lem2} as
\begin{align*}
\frac{\d\tilde g(k)}{\d k}
&=-\frac1\pi\,\Re\int_{-\infty}^{-2(3p^2+3p+2)/(p^2(1+p)^2)}\frac{p(p+1)\,\d v}{\sqrt{-(v+4)(p^2v+4)((p+1)^2v+4)}}
\\
&=-\frac1\pi\int_{-\infty}^{-4/p^2}\frac{p(p+1)\,\d v}{\sqrt{-(v+4)(p^2v+4)((p+1)^2v+4)}}.
\end{align*}
On the other hand, the substitution $t=(p^2v+4)/(p^2(v+4))$ in the integral of Lemma~\ref{lem4} leads to
$$
\frac{\d g(k)}{\d k}
=-\frac{1}\pi\int_{-\infty}^{-4/p^2}\frac{p(p+1)\,\d v}{\sqrt{-(v+4)(p^2v+4)((p+1)^2v+4)}},
$$
so that the first equality of the lemma follows.

For $0<k<8$, the substitution $t=-4/v$ in the integral of Lemma~\ref{lem3} results in
\begin{equation}
\frac{\d g(k)}{\d k}
=\frac1{2\pi}\int_{-\infty}^{-4}\frac{\d v}{\sqrt{-(v+4)(v^2+k(k-4)v+4k^2)}}.
\label{I1}
\end{equation}
This coincides with the representation of Lemma~\ref{lem2} for the range $0<k\le4$.
\end{proof}

\begin{remark}
Using the evaluation \eqref{S3} we can extend formula~\eqref{I1} to the interval $k>8$ as follows:
$$
\frac{\d g(k)}{\d k}
=\frac1{2\pi}\,\Re\int_{-\infty}^{-4}\frac{\d v}{\sqrt{-(v+4)(v^2+k(k-4)v+4k^2)}}.
$$
As in the case $4<k<8$, the resulting elliptic integral does not possess any clear relation with
the incomplete elliptic integral obtained for $\d\tilde g(k)/\d k$ in Lemma~\ref{lem2}.
\end{remark}

\begin{proof}[Proof of Theorem~\textup{\ref{th1}}]
Since both the Mahler measures $g(k)=\m(P_{2-k}(x,y))$ and $\tilde g(k)=\m(Q_k(X,Y))$
are continuous functions of real parameter~$k$, integrating the equalities of Lemma~\ref{lem5} imply
$$
\tilde g(k)=g(k)+c_- \quad\text{if}\; k\le-1,
\quad\text{and}\quad
\tilde g(k)=2g(k)+c_+ \quad\text{if}\; 0\le k\le4,
$$
where $c_-$ and $c_+$ are certain constants.

In~\eqref{Delta}, $\Delta_0(X)\le0$ on $|X|=1$, hence $|Y_1(X)|=|Y_2(X)|=1$ on the unit circle, and $\tilde g(0)=\m(\wt Q_0(X,Y))=0$.
Furthermore, $P_2(x,y)=(x+1)(y+1)(x+y)$ implying $g(0)=\m(P_2(x,y))=0$. Thus, $c_+=0$.

As $k\to\infty$,
\begin{align*}
\tilde g(k)
&=\m(Q_k(X,Y))
=\m\biggl(k\cdot\biggl((X+1)^2Y+\frac1k(Y^2+(X^4+1)Y+X^4)\biggr)\biggr)
\\
&=\log|k|+\m((X+1)^2Y)+O(1/k)
=\log|k|+O(1/k)
\end{align*}
and, similarly,
$$
g(k)=\m(P_{2-k}(x,y))\sim\log|k|+O(1/k).
$$
In particular, $\tilde g(k)-g(k)\to0$ as $k\to-\infty$ implying $c_-=0$.
\end{proof}

\begin{proof}[Proof of Theorem~\textup{\ref{th2}}]
The required evaluations follow from the formulae for $\m(P_0)$, $\m(P_{-2})$, $\m(P_1)$, $\m(P_4)$, $\m(P_6)$ and $\m(P_{10})$
obtained in \cite{Mel09,RV99,RZ12}.
\end{proof}

\section{Conclusion}
\label{s-concl}

Our proof of Theorem~\textup{\ref{th1}} makes use of two important features of the Mahler measures of polynomial families.
First, their derivatives with respect to the parameter of the family satisfy Picard--Fuchs differential equations\,---\,the property guaranteed
by the fact that the Mahler measures and their derivatives are periods. Second, when a family generically corresponds to curves of genus~2
whose Jacobian splits into the product of two elliptic curves, the derivative of the Mahler measure
reduces to (sometimes incomplete) elliptic integrals. Exploring this direction further we are able to establish
another general equality conjectured by Boyd~\cite{Boy98} between the Mahler measures of two hyperelliptic families
and even to relate them to the Mahler measures of an elliptic family.

\begin{theorem}
\label{th3}
Define
\begin{align*}
P_k(x,y)
&=(x^2+x+1)y^2+kx(x+1)y+x(x^2+x+1),
\\
Q_k(x,y)
&=(x^2+x+1)y^2
+(x^4+kx^3+(2k-4)x^2+kx+1)y
+x^2(x^2+x+1),
\\
R_k(x,y)
&=y^3-y+x^3-x+kxy,
\end{align*}
so that generically $P_k(x,y)=0$ and $Q_k(x,y)=0$ are families of genus~$2$ curves, while $R_k(x,y)=0$
is a family of elliptic curves for $k\ne0,\pm3$.
Then, for $k\in\mathbb R$ such that $|k|\ge16/(3\sqrt3)=3.0792\dots$, we have $\m(P_k)=\m(R_k)$. Furthermore, $\m(Q_{k+2})=\m(R_k)$ for $k\ge4$.
\end{theorem}

The proof of this theorem requires some other analytical tools and is presented in our paper~\cite{BeZu}.


\end{document}